\documentclass[12pt]{article}

\usepackage[centertags]{amsmath}
\usepackage{amsfonts}
\usepackage{amssymb}
\usepackage{amsthm}
\usepackage{color}

\newcommand{\pr}{\mbox{\sf P}}
\newcommand{\ex}{{\bf\sf E}}               
\newcommand{\var}{\mbox{\sf Var}}

\newcommand{\bP}{{\bf P}}               

\newcommand{\bu}{{\bf u}}               
\newcommand{\bb}{{\bf b}}               

\newcommand{\be}{{\bf e}}
\newcommand{\ba}{{\bf a}}

\newcommand{\al}{\alpha}                

\newcommand{\indic}{\mathbb{I}}

\newtheorem{thm}{Theorem}

\newtheorem{pro}{Proposition}

\newtheorem{rem}{Remark}

\newtheorem{ass}{Assumption}

\begin{document}

\title{
On Optimal Portfolios of Dynamic Resource Allocations
}

\date{February 24, 2017}

\author{Yingdong Lu, Siva Theja Maguluri, Mark S. Squillante\\Chai Wah Wu
\thanks{Y.~Lu, M.S.~Squillante and C.W.~Wu are with the Mathematical Sciences Department at the IBM Thomas J.\ Watson Research Center,
        {\tt\footnotesize \{yingdong,mss,cwwu\}@us.ibm.com};
	S.T.~Maguluri is with the H.\ Milton Stewart School of Industrial and Systems Engineering at the Georgia Institute of Technology,
        {\tt\footnotesize siva.theja@gatech.edu}}%
}
\renewcommand\footnotemark{}

\maketitle

\begin{abstract}
We consider the optimal allocation of generic resources among multiple generic entities of interest over a finite planning horizon,
where each entity generates stochastic returns
as a function of its
resource allocation
during each period.
The main objective is to maximize the expected return while at the same time
managing risk to an acceptable
level for each period.
We devise a general solution framework and establish how to obtain the optimal dynamic resource allocation.
\end{abstract}


\section{Introduction}
The trade-offs among risks and returns are central and fundamental to the planning and management of any collection of generic entities of interest to an organization.
At any given time, the organization may support various entities under its control with generic resources of interest.
The risks and returns associated with each entity as a function of the amount of allocated resources differ from one entity to the next, as well as across entities.
These risks and returns also evolve over time according to the dynamics of each entity, the dynamics across entities, and the dynamics of exogenous factors such as environmental changes.
Meanwhile, the organization seeks to optimize its own objectives within the context of various restrictions on what portfolio of resource allocation decisions and actions can be made and when.
A commonly used objective is to maximize the expected returns of a collection of entities of interest to the organization over a given planning horizon,
while simultaneously maintaining the risk exposure of its portfolio of resource allocation decisions and actions, usually represented by the variability of returns, within acceptable levels.
To ensure all objectives are achieved and restrictions are met, the decisions on resource allocations to each entity are periodically reviewed and adjusted over time.

This general dynamic resource allocation problem prominently arises for different types of entities across a broad spectrum of application domains.
Examples include applications in cybersecurity, pharmaceuticals, software products, and business management.
For cybersecurity applications,
the entities can represent different computing devices in distinct states of cybersecurity vulnerability;
the resources can include various forms of financial investments, people working on system security, and various cybersecurity applications (e.g., patch);
and the decisions can concern the amount of resources to be allocated to each device;
an analogous set of entities, resources and decisions exists for other forms of computer system management which may also include reliability and performance.
In the pharmaceutical industry,
the entities can represent research projects in different stages of development and trials;
the resources can include various forms of financial investments, scientists and researchers working on the projects, and various pharmaceutical compounds and materials;
and the decisions can concern the amount of resources to be allocated to each project.
Turning to the software industry for yet another example,
the entities can represent product offerings and releases of versions of product offerings;
the resources can include various forms of financial investments, software architects and programmers working on the products, and various equipment and infrastructure support;
and the decisions can concern the amount of resources to be allocated to each product.
Finally, as a general business management application,
the entities can represent different areas of the business organization such as development, marketing and sales;
the resources can include various forms of financial investments, people working in the areas, and various infrastructure support;
and the decisions can concern the amount of resources to be allocated to each area.
In each of these application domains, the organization may want to weight the benefits differently from one time period to the next such as
preferring long-term benefits even at the expense of short-term benefits or preferring short-term benefits even at the expense of long-term benefits.

Motivated by these real-world applications,
we consider a mathematical abstraction of the dynamic resource allocation problem and devise a general solution framework.
In particular, we mathematically model the resources and returns associated with every entity as quantities that can be measured in terms of certain forms of currency
which facilitate the interactions among different entities,
where this abstract notion of currency represents any combination of different types of resources such as financial investments, human resources, equipment, materials, and so on.
This mathematical abstraction also includes a fundamental connection between entities and financial instruments, even though they are very different from traditional financial instruments,
which leads to the representation of the returns from entities as random variables where such returns are functions of the amount of resources allocated to the entities and
evolve over time according to the intrinsic dynamics associated with the resource allocations and entities.
The allocation of resources are reviewed and adjusted over time to achieve maximum benefit as uncertainty is realized and information is updated.
In addition, at each time period, constraints are imposed on the variance of the return random variables such that only allocations satisfying these restrictions are deemed feasible.

Our mathematical models and analysis have certain degrees of resemblance to multi-period mean-variance models in the mathematical finance literature; see, e.g., \cite{LiChanNg,LiNg,ZhouLi}.
This includes an additional entity, with little or no risk, to serve as a reference point for the returns of the original set of entities,
as is standard in the literature~\cite{LiChanNg}.
In our context, the additional (last) entity can be cash or a Treasury note, thus representing an option that we do not need to invest all of the resources all of the time.
Alternatively, in other instances of our models, the additional (last) entity can be a benchmark entity that represents standards against which every entity will be evaluated.
There are also important connections between the resource allocations in our models and the investments in classical portfolio optimization, though with differences in semantics,
where there is an implicit assumption in our models that the returns from resource allocations at the end of one period are converted into the amount of resources available at the start of the next period.

At the same time, however,
there are significant differences between our models and those in the mathematical finance literature.
First, it is typical for the risk constraints to only be imposed on the terminal value in mathematical finance models.
This differs from our allocation of resources to entities together with their associated financial assets, where such operational aspects of the problem have a much larger impact
on the planning and management of entities and their returns, which in turn require that the associated risk factors be monitored and controlled at a much higher frequency.
Second, there appears to be an issue with a couple of key papers in the field~\cite{LiChanNg,LiNg} where, under the same model that we consider,
the amount of investment funds available from one period to the next are not properly restricted to the amount of wealth obtained up to that point in time;
i.e., an unlimited amount of ``borrowed'' funds are allowed at no cost.
This differs from our allocation of resources to entities, where allocation beyond current capacity can be much more difficult to do and therefore we explicitly enforce
the constraint of the problem formulation on the investment funds available in each period.
In addition, our models and analysis support the general case where borrowing is allowed at a given cost rate.
Our models and analysis further support the inclusion of weights to differentiate the contributions from one time period to the next.

We note that the research literature includes other mathematical models and methods to address the problem of dynamic resource allocation.
In particular, there have been several attempts to directly apply portfolio management models and methods in the planning and management of scientific and development projects.
One such approach~\cite{bardhan,gustaffson} is based on mathematical programming, including stochastic programming.
While some of these models and methods are quite powerful at incorporating constraints,
it is not surprising that one of the key features often ignored in these previous studies concerns the correlations among entities,
which usually causes the mathematical programming approach to become extremely complex and renders its computational complexity prohibitive~\cite{gustaffson}.
In contrast, as we will show, the mathematical models and methods devised in this paper can readily support many aspects of the various sources of uncertainty, including correlations among entities and their evolution over time.
Another approach to addressing the problem of project management is based on decision science methods, mostly involving various structured decision trees.
This approach has important limitations when the parameter and state spaces that have to be discretized become very large.
Our interests in this paper concern large-scale optimal dynamic allocation of resources among multiple entities of interest over a long-run planning horizon.

The remainder of the paper is organized as follows.
Our mathematical models are presented in Section~\ref{sec:models}, followed by our mathematical analysis in Section~\ref{sec:analysis}.
Extensions of our general solution framework are discussed in Section~\ref{sec:extensions}, followed by concluding remarks.

\section{Mathematical Models}
\label{sec:models}
Consider the allocation of a set of common resources among $n+1$ general entities over a discrete-time planning horizon comprising $T$ periods.
For each period $t=1, \ldots, T$, the amount of resources allocated to entity $i=1, \ldots, n+1$ is denoted by $u_i^t$.
Let $x_t$ denote the total return at the end of period $t$ where, with slight abuse of notation, such returns are implicitly converted into the amount of available resources at the start of period $t+1$.
The total amount of resources available at the beginning of the horizon is given by $x_0$, and thus $\sum_{i=1}^{n+1} u_i^1 = x_0$.
Within each period, random returns are generated proportional to the amount of resources allocated to each entity.
More specifically, let us define
\begin{itemize}
	\item $e_t := $ return of the $(n+1)$-st entity during time period $t$ per unit of resource allocated;
	\item $e^t_i := $ return of the $i$-th entity during time period $t$ per unit of resource allocated;
\end{itemize}
where these random variables (r.v.s) are independent over time.
Under these assumptions, $x_t$ satisfies the dynamics
\begin{align}
x_t &= e_t x_{t-1} + \sum_{i=1}^n (e^t_i -e_t) u_i^t , \qquad t=1, \ldots, T,
\label{eq:dynamics}
\end{align}
where the $(n+1)$-st entity is used as a point of reference and $u_{n+1}^t = x_{t-1} - \sum_{i=1}^n u_i^t$ by definition.

Defining the $n$-dimensional vectors $\bP_t$ and $\bu_t$ to be
\begin{eqnarray*}
\bP_t & := & [(e^t_1-e_t), (e^t_2-e_t), \ldots, (e^t_n-e_t)]^\prime , \\
\bu_t & := & (u_1^t, \ldots, u_n^t)^\prime ,
\end{eqnarray*}
respectively,
the system dynamics can be expressed as
\begin{align}
	\label{eqn:dynamic}
x_t &= e_t x_{t-1} + \bP^\prime_t  \bu_t.
\end{align}
We next define the $(n+1)$-dimensional vector
$\be^t := ( e_t, e_1^t, \ldots, e_n^t)^\prime$
and introduce the following assumption.

\begin{ass}
\label{ass:I}
$\ex[\be^t (\be^t)^\prime]$ is positive definite for all time periods $t=1, \ldots, T$.
\end{ass}

\begin{rem}
\label{rem:I}
While $\ex[\be^t (\be^t)^\prime]$ are positive semidefinite matrices by definition, Assumption~\ref{ass:I} guarantees they are not degenerate (i.e., $\be^t=0, \; \forall t$).
Further note that our model does not impose strong distributional assumptions on the returns; we only require that they have finite second moments.
\end{rem}

The $(n+1) \times (n+1)$ matrix $\ex[\be^t (\be^t)^\prime]$ then can be expressed as
\begin{align*}
\left[ \begin{array}{cccc} \ex[e_t^2] & \ex[e_te^t_1] & \ldots  & \ex[e_te^t_n] \\ \ex[e_1^te_t] & \ex[(e^t_1)^2] & \ldots  & \ex[e_1^te^t_n] \\ && \ldots & \\
\ex[e_n^te_t] & \ex[e_n^te^t_1] & \ldots  & \ex[(e^t_n)^2] \end{array}
\right] .
\end{align*}
Assumption~\ref{ass:I} implies that
\begin{align*}
&\left[ \begin{array}{cc} \ex[(e_t)^2] & \ex[e_t\bP^\prime_t] \\ \ex[e_t\bP_t] & \ex[\bP_t\bP^\prime_t] \end{array}\right] = \\ &  \left[\begin{array}{cccc} 1 & 0 &\ldots & 0 \\  -1 & 1 &\ldots & 0 \\ \ldots & \ldots &\ldots & \ldots \\  -1 & 0 &\ldots & 1 \end{array}
\right]\ex[\be^t (\be^t)^\prime] \left[\begin{array}{cccc} 1 & -1 &\ldots & -1 \\  0 & 1 &\ldots & 0 \\ \ldots & \ldots &\ldots & \ldots \\  0 & 0 &\ldots & 1 \end{array}
\right]
\end{align*}
is also positive definite.
In addition, we have
\begin{eqnarray}
\ex[\bP_t\bP^\prime_t] & > & 0, \qquad \forall t, \label{eqn:condition1} \\
\ex[(e_t)^2] - \ex[e_t\bP^\prime_t]\ex^{-1}[\bP_t\bP^\prime_t]\ex[e_t\bP_t] & > & 0, \qquad \forall t. \qquad \label{eqn:condition2}
\end{eqnarray}

Our resource allocation goal is to maximize the expected return over time, while maintaining the variance of return within an acceptable range.
This is a natural model for many different problem instances
across a wide variety of applications.
The progress of resource allocations are often closely monitored and periodically reviewed so that they can be adjusted to yield the best return.
Hence, the returns and risks of each period need to be factored into the decision making process as forms of operational measurements.

Mathematically, our resource allocation planning can be formulated as the following stochastic optimization problem
\begin{align}
	\label{eqn:orignal_formulation}
\max_{(\bu_1,\ldots,\bu_T)} \quad & \sum_{t=1}^Tw_t\ex[x_t]\\
\mbox{s.t.} \quad  &  \var[x_t]\le \al_t, \qquad t=1, \ldots, T , \label{eqn:original_constraints}
\end{align}
where the weights $w_t$ allow differentiation among the contributions from distinct periods,
and the pre-specified risk tolerances $\al_t$
reflect levels of risk deemed to be acceptable.
Let us denote the set of optimal allocations as $\Pi_1(\mathbf{w},\mathbf{\alpha})$.

This type of portfolio optimization problem is often investigated through related Lagrangian relaxations.
More specifically, solving this problem is equivalent to solving
\begin{align}
	\label{eqn:second_formulation}
\max_{(\bu_1,\ldots,\bu_T)} \quad & \sum_{t=1}^Tw_t\ex[x_t]-y_t \var[x_t] 
\end{align}
for some vectors $\mathbf{y}$.
Denote the set of optimal allocations by $\Pi_2(\mathbf{w},\mathbf{y})$.
This problem remains challenging to solve because of the variance term.
Since variance is not linear, and since $x_t$ evolves according to \eqref{eq:dynamics}, the variance term here couples the decisions across the time slots.
However, as in \cite{LiChanNg,LiNg}, the problem can be solved by considering a simpler set of problems that also renders the optimal solution of \eqref{eqn:second_formulation}.
To this end, we consider the general problem
\begin{align}
		\label{eqn:equivalent_formulation}
\max_{(\bu_1,\ldots,\bu_T)} \quad & \sum_{t=1}^T\ex [a_tx_t-b_t x_t^2].
\end{align}
Let us denote the set of optimal allocations as $\Pi_3(\ba,\bb)$.
We henceforth focus on the solution of \eqref{eqn:equivalent_formulation} with respect to $a_t$ and $b_t$ as a result of the following proposition.
\begin{pro}
	\label{pro:equivalence}
Any solution to the problem \eqref{eqn:second_formulation} is also a solution to the problem \eqref{eqn:equivalent_formulation} for some vectors $\ba$ and $\bb$,
i.e., $\Pi_2(\mathbf{w},\mathbf{y})\subseteq \Pi_3(\ba,\bb)$ for some vectors $\ba$ and $\bb$.
\end{pro}
\begin{proof}
The proof of this proposition is by contradiction and similar to that of Theorem~1 in \cite{LiNg}, and thus we only provide a brief summary here.
Let $U(\cdot)$ denote the objective in \eqref{eqn:second_formulation}.
Suppose there is a sequence of optimal decisions for \eqref{eqn:second_formulation}, $\pi^\ast=(\bu^\ast_1,\ldots,\bu^\ast_T) \in \Pi_2(\mathbf{w},\mathbf{y})$,
and let $\mathbf{x}^\ast$ denote the corresponding returns. 
Set $b_t= y_t$ and $a_t=w_t+2y_tE[x^\ast_t]$ and further suppose that   $(\bu^\ast_1,\ldots,\bu^\ast_T) \notin \Pi_3(\ba,\bb)$.
Note that we selected $a_t=\frac{\partial{U}}{\partial{E[x_t]}}$ under the optimal decisions $\pi^\ast$. 
There exists a $\tilde{\pi}=(\tilde{\bu}_1,\ldots,\tilde{\bu}_T)$ that is an optimal solution of  \eqref{eqn:equivalent_formulation}.
From this fact together with the fact that $U(\cdot)$ is convex in $E[x_t]$ and $E[x^2_t]$ for all $t$,
one can conclude that using the decisions $\tilde{\pi}$ leads to a better value of $U(\cdot)$ than the decisions $\pi^\ast$, which contradicts the optimality of $\pi^\ast$.
\end{proof}

\section{Mathematical Analysis}
\label{sec:analysis}
We now present our analysis of the above stochastic optimization problem.
First, to illustrate the basic approach, we present a detailed analysis of the one-dimensional case where many quantities can be explicitly calculated.
Then, we proceed to provide our solution for the general problem, followed by extensions of our analysis to support acquisition of additional resources in each period at a cost.

\subsection{One-dimensional Case}

Consider the case of $n=1$.
From \eqref{eqn:dynamic}, the total return for any period $t=1,\ldots, T$ is given by
$$R_t = e_t x_{t-1} + P_t u_t = e_1^t u_t + e_t (x_{t-1} - u_t),$$ with $u_t$ denoting the amount of resource allocation.
The r.v.\ $e_1^t$ has first two moments $p_{1,t}$ and $p_{2,t}$, whereas the return of the second entity $e_t$ is assumed
to be riskless with mean $r_0>0$ and variance $0$.
This then renders the stochastic optimization problem
\begin{equation*}
\max_{(u_1,\ldots,u_T)} \quad \sum_{t=1}^T [a_t \mu_t - b_t \nu_t] ,
\end{equation*}
where $a_t, b_t \ge 0$ are given parameters, $\mu_t = \ex[R_t]$, and $\nu_t = \ex[R^2_t].$

Let us next consider the solution of this stochastic dynamic program.
At the beginning of period $T$, since the total amount of available resources is $x_{T-1}$,
we only need to determine the resource allocations $u_T$ for time period $T$ that maximize
$	J_T \; = \; a_T \mu_T - b_T \nu_T.$
If the total allocation is $y$, the objective can be rewritten as
\begin{align*}
J_T & = 	a_T [y p_{1,T}+r_0(x_{T-1}-y)] \\ & -  b_T [y^2p_{2,T}+ r_0^2(x_{T-1}-y)^2+2r_0p_{1,T}y(x_{T-1}-y)] .
\end{align*}
Upon taking the derivative with respect to $y$, we obtain
\begin{align*}
\frac{\partial J_T}{\partial y}	
 = & \; [a_T (p_{1,T}-r_0) + 2b_Tx_{T-1}(r_0^2-r_0p_{1,T})] \\ &- 2b_T(p_{2,T}+r_0^2-2r_0p_{1,T})y.
\end{align*}
Hence, the stationary point is given by
\begin{align*}
	y^\ast \; = \; \frac{a_T (p_{1,T}-r_0) + 2b_Tx_{T-1}(r_0^2-r_0p_{1,T})}{2b_T(p_{2,T}+r_0^2-2r_0p_{1,T})}.
\end{align*}

In contrast to~\cite{LiChanNg,LiNg}, we enforce the available resource constraint $x_{T-1}$ on allocation decisions.
More specifically, on one side when $y^\ast > x_{T-1}$, we have
\begin{align*}
	\frac{a_T (p_{1,T}-r_0) + 2b_Tx_{T-1}(r_0^2-r_0p_{1,T})}{2b_T(p_{2,T}+r_0^2-2r_0p_{1,T})} \; > \; x_{T-1}
\end{align*}
which is equivalent to
\begin{align*}
	x_{T-1} \; < \; \frac{a_T (p_{1,T}-r_0)}{2b_T(p_{2,T}-r_0p_{1,T})} .
\end{align*}
Then the stationary point is out of reach and the maximum is obtained at $x_{T-1}$.
The optimal policy therefore allocates all resources to the first entity,
thus rendering
\begin{align*}
	J_T (x_{T-1}) \; = \; a_T x_{T-1} p_{1,T}-b_T x_{T-1}^2p_{2,T}.
\end{align*}
Observe that $J_T (x_{T-1})$ has a decreasing derivative as a function of $x_{T-1}$ and
that the right derivative of $J_T (x_{T-1})$ evaluated at $$x_{T-1} = \frac{a_T (p_{1,T}-r_0)}{2b_T(p_{2,T}-r_0p_{1,T})}$$ is equal to, after some simplification,
\begin{align*}
\frac{p_{2,T}-p^2_{1,T}}{p_{2,T}-r_0p_{1,T}}a_Tr_0 .
\end{align*}
On the other hand, if $$x_{T-1} \geq \frac{a_T (p_{1,T}-r_0)}{2b_Tp_{2,T}},$$ then the optimal policy
allocates $y^\ast$ resources to the first entity,
in which case the value function bears the form
\begin{align*}
	J_T (x_{T-1})= & a_T [y^\ast p_{1,T}+r_0(x_{T-1}-y^\ast)]-b_T [(y^\ast)^2p_{2,T} \\& + r_0^2(x_{T-1}-y^\ast)^2+2r_0p_{1,T}y^\ast(x_{T-1}-y^\ast)].
\end{align*}
Straightforward calculations show that $J_T (x_{T-1})$ has a decreasing derivative with respect to $x_{T-1}$
and the left derivative at $$x_{T-1} = \frac{a_T (p_{1,T}-r_0)}{2b_T(p_{2,T}-r_0p_{1,T})}$$ is also equal to
\begin{align*}
 \frac{p_{2,T}-p^2_{1,T}}{p_{2,T}-r_0p_{1,T}}a_Tr_0 .
\end{align*}
As a result of these two cases, we can conclude that the value function $J_T (x_{T-1})$ is a concave function of $x_{T-1}$.

With $x_{T-2}$ resources available at the beginning of period $T-1$,
if the allocation to the first entity is set at $y\le x_{T-2}$, then the objective can be expressed as
\begin{align*}
	\max \; J_{T-1}=a_{T-1} \mu_{T-1} - b_{T-1} \nu_{T-1}+ \ex[J_T(x_{T-1})] .
\end{align*}
Equivalently, we have
\begin{align*}
	\max_y \quad & a_{T-1} [y p_{1,T-1}+r_0(x_{T-2}-y)] -b_{T-1}[ y^2p_{2,T-1} + r_0^2( x_{T-2}-y)^2\\ & + 2r_0p_{1,T-1}(x_{T-2}-y)] + \ex[J_T(r_0(x_{T-2}-y)+ e_1^{T-1}y)].
\end{align*}
Meanwhile, from the above analysis, we know that
\begin{align*}
	&\ex[J_T(r_0(x_{T-2}-y)+ e_1^{T-1}y)]\\  = & \ex[ \{a_T [(r_0(x_{T-2}-y)+ e_1^{T-1}y)p_{1,T} ] \\& - b_T [{(r_0(x_{T-2}-y)+ e_1^{T-1}y)}^2 p_{2,T}]\} \times  \indic\{G_{T-1}\}]\\ &+ \ex[ \{a_T [(y^\ast)p_{1,T} + r_0((x_{T-2}-y)+ e_1^{T-1}y-y^\ast)] \\ &- b_T [{(y^\ast)}^2 p_{2,T} + r_0^2 {(x_{T-2}-y+ e_1^{T-1}y-y^\ast)}^2] \}  \times \indic\{G^c_{T-1}\}]
\end{align*}
with
\begin{align*}
	G_{T-1} := \left\{ r_0(x_{T-2}-y)+ e_1^{T-1}y\le \frac{a_T (p_{1,T}-r_0)}{2b_T(p_{2,T}-r_0p_{1,T})}\right\} ,
\end{align*}
where $\indic\{Z\}$ is the indicator function for event $Z$.

The above quantity, of course, can be easily calculated for any distributions with finite first two moments $p_{1,t}$ and $p_{2,t}$ for the independent r.v.s $e_1^{t}$.
Once again, in contrast to~\cite{LiChanNg,LiNg}, we explicitly consider the constraint of available resources $x_{T-2}$ on allocation decisions.
Hence, we can repeat the process for the last period by finding the stationary point, optimally allocating the amount of resources $(x_{T-2}\wedge y^\ast)$ to the first entity,
and then exploiting the properties of the function to show that the value function remains concave.

In general, for any time period $t$, with initial availability of $x_{t-1}$ resources, the objective becomes
\begin{align*}
	\max_y \quad & a_{t} [y p_{1,t}+r_0(x_{t-1}-y)] -b_{t} [y^2p_{2,t}+r_0^2(x_{t-1}-y)^2\\ & +2r_0p_{1,t}y(x_{t-1}-y)] + \ex[J_{t+1}(r_0(x_{t-1}-y)+ e_1^{t}y)]
\end{align*}
where $\ex[J_{t+1}(r_0(x_{t-1}-y)+ e_1^{t}y)]$ can be readily calculated as demonstrated above.
Then we can obtain the critical value and the optimal policy will again be to allocate the minimum between this threshold and $x_{t-1}$.
To guarantee obtaining the critical value efficiently, which is key in the derivation for the one-dimensional case,
we will establish the preservation of the concavity of the objective for any period.
Our approach is based on inductive arguments, where the initial case has been discussed fully above.
Now, assume the concavity of the objective function holds for any period $t+1, \ldots, T$.

At the beginning of period $t$, the amount of available resources is $x_{t-1}$ such that $y\le x_{t-1}$ resources can be allocated to the first entity which yields
a total return of $$e_1^ty + J_{t+1}(e_1^ty+ r_0( x_{t-1}-y)).$$
Hence, the optimal decision for period $t$ is determined through the following problem
\begin{align}
	\max_y \quad & a_t p_{1,t} y + a_t r_0(x_{t-1}-y) - b_t p_{2,t} y^2
	 - b_t r_0^2{(x_{t-1}-y)}^2 \nonumber \\ & -2b_tr_0p_1,ty(x_{t-1}-y) + \ex[J_{t+1}(e_1^ty+ r_0(x_{t-1}-y))].
\label{obj:gen:t}
\end{align}
Taking the derivative of the objective function with respect to $y$, we obtain
\begin{align}
	& a_t(p_{1,t}-r_0) + 2b_t x_{t-1} (r_0^2 -r_0p_{1,t})  - 2b_t(p_{2,t}+r_0^2-2r_0p_{1,t})y \nonumber \\ &\qquad\qquad + \frac{\partial }{\partial y} \ex[J_{t+1}(e_1^ty+  r_0(x_{t-1}-y))].
\label{deriv:gen:t}
\end{align}
Let $y^\ast$ denote the critical value for this function.
Further, since $y^\ast$ is a function of $x_{t-1}$, let us denote by $x_{t-1}^\ast$ the solution to $y^\ast (x_{t-1})=x_{t-1}$.
Then, if $x_{t-1} < x_{t-1}^\ast$, the optimal decision for period $t$ is to allocate all $x_{t-1}$ available resources to the first entity;
otherwise, the optimal decision is to allocate $x_{t-1}^\ast$ resources to the first entity.

Hence, when $x_{t-1} < x_{t-1}^\ast$, the value function can be expressed as
\begin{align*}
	a_t p_{1,t} x_{t-1} - b_t p_{2,t} x_{t-1}^2 + \ex[J_{t+1}(e_1^tx_{t-1})] ,
\end{align*}
which is clearly a concave function of $x_{t-1}$.
Moreover, the left derivative evaluated at $x_{t-1}= y^\ast$ is given by
\begin{align}
\label{eqn:leftD}
	a_t p_{1,t} - 2 b_t p_{2,t} y^\ast + \frac{\partial }{\partial y} \ex[J_{t+1}(e_1^t y^\ast)] .
\end{align}
By definition of $y^\ast$, the above expression is equivalent to
\begin{align*}
	1- \frac{\partial }{\partial y}\ex[J_{t+1}(e_1^ty+ r_0( x_{t-1}-y))]+\frac{\partial }{\partial y}\ex[J_{t+1}(e_1^ty^\ast)] ,
\end{align*}
which tends to approach $1$ as $x_{t-1}\rightarrow x_{t-1}^\ast$.
In summary, the derivative remains positive and decreases when $x_{t-1}< x_{t-1}^\ast$.

On the other hand, when $x_{t-1}\ge x_{t-1}^\ast$, the value function and its derivative are given by \eqref{obj:gen:t} and \eqref{deriv:gen:t}, respectively.
From the definition of $x_{t-1}^\ast$, we know that the derivative is less than or equal to \eqref{eqn:leftD} and non-increasing.
Concavity is therefore preserved and the desired result follows by induction.

\subsection{General Case}
We proceed along similar lines in our development of an algorithmic framework to solve the general problem.
Starting at period $T$, with initial available resources $x_{T-1}$, the objective that is to be maximized has the following form
\begin{align*}
J_{T}  = & \ex[-b_Tx_T^2 + a_T x_T \; | \; x_{T-1}].
\end{align*}
Substituting the expression \eqref{eqn:dynamic} for $x_T$, we obtain
\begin{align*}
J_{T} 
= & \{-b_T\ex[e^2_{T}]x^2_{T-1} + a_T\ex[e_{T}] x_{T-1}\}  + \{a_T\ex[\bP^\prime_{T}]-2b_T\ex[e_{T}\bP^\prime_{T}]x_{T-1}\}\bu_{T} \\ & - b_T \bu_{T}^\prime \ex[\bP_{T}\bP^\prime_{T}] \bu_{T} ,
\end{align*}
and therefore
\begin{align*}
\frac{\partial J_{T}(\bu_{T} \;|\; x_{T-1})}{\partial \bu_{T}} = &
 a_T\ex[\bP^\prime_{T}]-2b_T\ex[e_{T}\bP^\prime_{T}]x_{T-1} -2b_T\ex[\bP_{T}\bP^\prime_{T}]\bu_{T} .
\end{align*}
Upon examining the system of equations
$$\frac{\partial J_{T}(\bu_{T} \;|\; x_{T})}{\partial \bu_{T}} \; = \; 0,$$
we derive the solution
\begin{equation}
\bu_{T}^\ast = \ex^{-1}[\bP_{T}\bP^\prime_{T}] \{ a_T/(2b_T)\ex[\bP_{T}] -\ex[e_{T}\bP^\prime_{T}]x_{T-1} \} . \label{eqn:sol1}
\end{equation}

Now, if ${\bf 1}^\prime \bu_{T}^\ast \le  x_{T-1}$ with ${\bf 1}$ being the $n$-dimensional column vector of all ones,
i.e., the initial amount of available resources is large enough to accommodate the allocation corresponding to the solution \eqref{eqn:sol1},
then $\bu_{T}^\ast$ is in fact the optimal resource allocation decision.
Otherwise, we need to introduce the constraint ${\bf 1}^\prime \bu_{T} =  x_{T-1}$, as well as a multiplier $\nu$.
Applying this to the optimization problem, we obtain
\begin{align*}
\bu_{T}   = &  \ex^{-1}[\bP_{T}\bP^\prime_{T}] c_T -\nu,  \quad {\bf 1}^\prime \bu_{T}  = x_{T-1},
\end{align*}
with $$c_T:=\frac{a_T}{2b_T}\ex[\bP_{T}]-\ex[e_{T}\bP^\prime_{T}]x_{T-1},$$ which renders
\begin{align*}
\nu &= \frac{1}{n}\left\{{\bf 1}^\prime\ex^{-1}[\bP_{T}\bP^\prime_{T}] c_T -x_{T-1}\right\}
\end{align*}
and therefore
\begin{align}
\bu_{T}^\ast = & \ex^{-1}[\bP_{T}\bP^\prime_{T}] c_T  -\frac{1}{n}\left\{{\bf 1}^\prime\ex^{-1}[\bP_{T}\bP^\prime_{T}] c_T -x_{T-1}\right\}.	\label{eqn:sol2}
\end{align}
Upon combining the solutions in \eqref{eqn:sol1} and \eqref{eqn:sol2}, we have
\begin{align}
\bu_{T}^\ast = & \ex^{-1}[\bP_{T}\bP^\prime_{T}] c_T -\frac{1}{n}\left\{{\bf 1}^\prime\ex^{-1}[\bP_{T}\bP^\prime_{T}] c_T -x_{T-1}\right\}^+ .
\label{eqn:opt_allocation_T}
\end{align}
Substituting this into
the objective function, we derive the optimal cost-to-go for a given $x_{T-1}$ to be
\begin{align*}
& J_{T}^\ast (x_{T-1}) = \\
& -b_T\Big[\ex[e^2_{T}]-2\Big(1-\frac{1}{n}A_{T}\Big)\ex[e_{T}\bP^\prime_{T}]\ex^{-1}[\bP_{T}\bP^\prime_{T}]  \times \ex[e_{T}\bP^\prime_{T}] -A_{T} \Big]x_{T-1}^2 \\
& + a_T\Big[ \ex[e_{T}]-\Big(1-\frac{1}{n}A_{T}\Big)\ex[\bP^\prime_{T}]\ex^{-1}[\bP_{T}\bP^\prime_{T}]\ex[e_{T}\bP^\prime_{T}]\Big]  x_{T-1} \\
&  + \Big(1-\frac{1}{n}A_{T}\Big) \frac{a_T^2}{4b_T} \ex[\bP^\prime_{T}]\ex^{-1}[\bP_{T}\bP^\prime_{T}]\ex[\bP_{T}]
\end{align*}
where
$$
A_{T} =  \indic\big\{ \ex^{-1}[\bP_{T}\bP^\prime_{T}] \big[c_T-{\bf 1}x_{T-1}\big] \ge 0 \big\} .
$$
Observe that the above expression for $J_{T}^\ast (x_{T-1})$ is a joint combination of two quadratic functions of $x_{T-1}$.
Furthermore, by the following known result in convex analysis (see, e.g., \cite{RockafellarBook}), $J_{T}^\ast (x_{T-1})$ remains a concave function of $x_{T-1}$.
\begin{thm}
	The solution of a convex programming problem with one knapsack constraint is convex with respect to the constrained value.
\end{thm}

Next, proceeding to period $T-1$ with an initial amount of $x_{T-2}$ resources, the objective can be expressed as
\begin{equation*}
	\max_{\bu_{T-1}} \quad a_{T-1} \ex[x_{T-1}]-b_{T-1} \ex[x^2_{T-1}] + \ex[J_{T}^\ast(x_{T-1})] .
\end{equation*}
We can write the objective function at time $T-1$ in the following nominal form
\begin{align*}
	 {\hat a}_{T-1} x_{T-1} -{\hat b}_{T-1} x^2_{T-1} + {\hat \gamma}_{T-1}  \ex[\bP^\prime_{T}]\ex^{-1}[\bP_{T}\bP^\prime_{T}]\ex[\bP_{T}]
\end{align*}
with
\begin{align*}
{\hat a}_{T-1} & = a_{T-1}+a_{T} \Big[ \ex[e_{T}]-\Big(1-\frac{1}{n}A_{T}\Big)\ex[\bP^\prime_{T}]  \ex^{-1}[\bP_{T}\bP^\prime_{T}]\ex[e_{T}\bP^\prime_{T}] \Big], \\
{\hat b}_{T-1} & = b_{T-1}+b_{T} \Big[ \ex[e^2_{T}]-2\Big(1-\frac{1}{n}A_{T}\Big)\ex[e_{T}\bP^\prime_{T}]  \ex^{-1}[\bP_{T}\bP^\prime_{T}]\ex[e_{T}\bP^\prime_{T}] -A_{T} \Big], \\
{\hat\gamma}_{T-1} & =\Big(1-\frac{1}{n}A_{T}\Big)\frac{a_T^2}{4b_{T}} .
\end{align*}
Once again, we follow the same approach as presented above.
First, we obtain the derivatives with respect to the allocation $\bu_{T-1}$ at beginning of time period $T-1$ and solve for the stationary point.
Note that the derivatives will be piecewise linear functions of $\bu_{T-1}$.
Then we check whether the allocation exceeds the initial amount of available resources;
whenever this is the case, we introduce the Lagrangian multiplier and obtain the corresponding optimal allocation.

This method can be carried out for any time period.
In general, for time period $t$, the objective will be
\begin{align*}
	 {\hat a}_tx_t -{\hat b}_t x_t^2 + \sum_{s=t}^{T-1} {\hat \gamma}_t  \ex[\bP^\prime_{t+1}]\ex^{-1}[\bP_{t+1}\bP^\prime_{t+1}]\ex[\bP_{t+1}]
\end{align*}
with
\begin{align*}
{\hat a}_{t}  = & a_t+  a_{t+1} \Big[ \ex[e_{t+1}]-\Big(1-\frac{1}{n}A_{t+1}\Big)  \ex[\bP^\prime_{t+1}]\ex^{-1}[\bP_{t+1}\bP^\prime_{t+1}]\ex[e_{t+1}\bP^\prime_{t+1}] \Big], \\
{\hat b}_{t} = & b_t+ b_{t+1}\Big[\ex[ e^2_{t+1}]-2\Big(1-\frac{1}{n}A_{t+1}\Big)\ex[e_{t+1}\bP^\prime_{t+1}] \\ & \times \ex^{-1}[\bP_{t+1}\bP^\prime_{t+1}]\ex[e_{t+1}\bP^\prime_{t+1}] -A_{t+1} \Big], \\
{\hat \gamma}_{t}  = &\Big(1-\frac{1}{n}A_{t+1}\Big)\frac{{\hat a}_{t+1}^2}{4{\hat b}_{t+1}},
\end{align*}
and
$$
	A_{t} \; = \;  \indic \big\{ \ex^{-1}[\bP_{t}\bP^\prime_{t}] \big[{\hat a}_t/(2{\hat b}_t)\ex[\bP_{t}] -(\ex[e_{t}\bP^\prime_{t}]+{\bf 1})x_{t-1}\big]\ge 0\big\}.
$$
The optimal allocation then takes the form
\begin{align*}
\bu_{t}^\ast  = & \ex^{-1}[\bP_{t}\bP^\prime_{t}] \bigg[\frac{{\hat a}_{t+1}}{2{\hat b}_{t+1}}\ex[\bP_{t}]-\ex[e_{t}\bP^\prime_{t}]x_{t}\bigg] \\ &- \frac{1}{n}\bigg\{\ex^{-1}[\bP_{t}\bP^\prime_{t}] \left[\frac{{\hat a}_{t+1}}{2{\hat b}_{t+1}}\ex[\bP_{t}]-\ex[e_{t}\bP^\prime_{t}]x_{t}\bigg]-x_{t}\right\}^+ .
\end{align*}

Another important fact is that at time $t$,  the objective function will be the joint combination of at most $T-t+1$ quadratic functions.
As seen at each step of the backward induction, we are introducing at most one break point, which increases the number of quadratic functions by at most one.
This yields the following result.
\begin{pro}
The objective function at time $t$ is a piecewise quadratic function with at most $T-t+1$ pieces.
\end{pro}

Our algorithmic framework for solving the stochastic dynamic programming problem can be summarized as follows:
\begin{itemize}
	\item[] At the beginning of period $t$, the initial amount of resources is $x_{t-1}$ and the optimal allocation of resources is given by the solution of the convex programming problem for period $t$.  Moreover, a function $f_t(x_{t-1})$ is obtained as the localization of the objective function at period $t$, where this function is increasing and concave.
\end{itemize}

\subsection{Extensions on Resource Flexibility}
Our general solution framework thus far has strictly enforced the available resource constraint $x_{t}$ on allocation decisions at any time $t$.
We now relax this constraint and consider the more general case that allows additional resources to be acquired at a per-unit cost of ${\tilde c}_t$.
While this introduces a more complex cost structure, the formulation and solution gain more flexibility and agility in allocating resources for each time period.
We therefore adapt our algorithmic framework to address the following objective function at each time $t$
$${\tilde a}_tx_t -{\tilde b}_t x_t^2 - {\tilde c}_t[{\bf 1}'\bu_t -x_{t-1}]^+  + \sum_{s=t}^{T-1} {\tilde \gamma}_t  \ex[\bP^\prime_{t+1}]\ex^{-1}[\bP_{t+1}\bP^\prime_{t+1}]\ex[\bP_{t+1}] ,
$$
where 
\begin{align*}
{\tilde a}_{t}  = & a_t+  a_{t+1} \Big\{ \ex[e_{t+1}]-(1-{\tilde c}_tA_{t+1})  \ex[\bP^\prime_{t+1}]\ex^{-1}[\bP_{t+1}\bP^\prime_{t+1}]\ex[e_{t+1}\bP^\prime_{t+1}] \Big\}, \\
{\tilde b}_{t} = & b_t+ b_{t+1}\Big\{\ex[ e^2_{t+1}]-2(1-{\tilde c}_tA_{t+1})\ex[e_{t+1}\bP^\prime_{t+1}] \\ & \times \ex^{-1}[\bP_{t+1}\bP^\prime_{t+1}]\ex[e_{t+1}\bP^\prime_{t+1}] -A_{t+1} \Big\}\, \\
{\tilde \gamma}_{t}  = &(1-{\tilde c}_tA_{t+1})\frac{{\tilde a}_{t+1}^2}{4{\tilde b}_{t+1}} .
\end{align*}
The optimal solution then takes the analogous form
\begin{align*}
\bu_{t}^\ast  = & \ex^{-1}[\bP_{t}\bP^\prime_{t}] \bigg[\frac{{\tilde a}_{t+1}}{2{\tilde b}_{t+1}}\ex[\bP_{t}]-\ex[e_{t}\bP^\prime_{t}]x_{t}\bigg].
\end{align*}

\section{Extensions on General Formulations}
\label{sec:extensions}
The stochastic dynamic programming problem considered above is most related to the various real-world applications motivating this study.
Our algorithmic solution framework derived herein, however, is quite general and can be applied to a much wider range of stochastic dynamic program formulations,
including those in which different objective functions are used to reflect distinct aspects of the problem.

One performance criteria often used to measure the success of entities of interest concerns the likelihood of achieving certain pre-specified targets.
This class of problems can be readily modeled by a chance constraint, or equivalently by an objective function with a chance element.
More specifically, we can have the alternative objective function
\begin{align*}
\max_{(u_1,\ldots,u_T)} \quad & \sum_{t=1}^T w_t\pr[x_t > d_t] ,
\end{align*}
where the $d_t$ are any type of target return values that the decision makers usually consider to be a primary indicator of the progress and success for each of the various entities,
and the $w_t$ are user-specified weights that differentiate the importance of each chance element over time.
From Tchebyshev's inequality, we know that, when $x_t$ is large,
\begin{align*}
\pr[x_t> d_t]\le \frac{\var[x_t]}{(d_t-\ex[x_t])^2} ,
\end{align*}
and therefore we can again obtain a proper weighted objective function of the mean and variance of $x_t$ to formulate an equivalent problem to which our solution framework applies.

More generally, variations of our algorithmic solution framework can be applied to address the following generic class of dynamic utility maximization problems:
\begin{align*}
\max_{(u_1,\ldots,u_T)} \quad & \sum_{t=1}^T U_t(\ex[x_t], \var[x_t])
\end{align*}
where the $U_t(\cdot, \cdot)$ represent various utility functions of the mean and variance of the returns which exhibit properties that are increasing in $\ex[x_t]$ and decreasing in $\var[x_t]$.
This includes a wide range of risk measures, especially those exhibiting convexity, such as the entire class of value-at-risk measures.
More recent measures such as least partial moments can also be incorporated within our general approach.

Lastly, there are many studies related to continuous time mean-variance portfolio analysis that are also of interest.
Our algorithmic solution framework derived herein can be similarly applied to address this general class of problems.

\section{Conclusions}

We devised a general solution framework to support the dynamic allocation of generic resources among generic entities of interest over a finite horizon during which entities generate random returns.
With a primary objective to maximize the expected return while restricting the risk to a manageable level for each period,
we derived an algorithmic solution framework for determining the optimal portfolio of resource allocations.
Our general solution framework can be applied to address the planning and management of entities of interest across a broad spectrum of application domains.

\end{document}